\documentclass[12pt,reqno,tbtags,draft]{amsart}
\usepackage{amssymb}
\usepackage{geometry}
\usepackage{ifdraft}
\usepackage{url}
\usepackage[square,numbers]{natbib}

\numberwithin{equation}{section}

\renewcommand\le{\leqslant}
\renewcommand\ge{\geqslant}

\allowdisplaybreaks

\newtheorem{theorem}{Theorem}[section]
\newtheorem{lemma}[theorem]{Lemma}

\theoremstyle{definition}
\newtheorem{example}[theorem]{Example}
\newtheorem{construction}[theorem]{Construction}

\newtheorem{problem}[theorem]{Problem}
\newtheorem{remark}[theorem]{Remark}

\newtheorem*{ack}{Acknowledgement}

\theoremstyle{remark}

\newenvironment{romenumerate}[1][0pt]{
\addtolength{\leftmargini}{#1}\begin{enumerate}
 }{\end{enumerate}}

\newcounter{oldenumi}
{\setcounter{oldenumi}{\value{enumi}}
\begin{romenumerate} \setcounter{enumi}{\value{oldenumi}}}
{\end{romenumerate}}

\newcounter{thmenumerate}

\newcounter{romxenumerate}   

\newcounter{xenumerate}   

\newcommand{\refT}[1]{Theorem~\ref{#1}}

\newcommand{\refCon}[1]{Construction~\ref{#1}}
\newcommand{\refL}[1]{Lemma~\ref{#1}}

\newcommand{\refS}[1]{Section~\ref{#1}}

\newcommand{\refE}[1]{Example~\ref{#1}}

\newcommand{\refand}[2]{\ref{#1} and~\ref{#2}}

\newcommand\marginal[1]{\ifdraft
{\marginpar[\raggedleft\tiny #1]{\raggedright\tiny #1}}
{\message{ERROR marginal requires draft option}}}

\newcommand\REM[1]{{\raggedright\texttt{[#1]}\par\marginal{XXX}}}

\newenvironment{comment}{\setbox0=\vbox\bgroup}{\egroup} 

\begingroup
  \divide\count255 by 60
  \multiply\count255 by -60
  \advance\count255 by \time
  \ifnum \count255 < 10 \xdef\klockan{\the\count1.0\the\count255}
  \else\xdef\klockan{\the\count1.\the\count255}\fi
\endgroup

\newcommand{\prodjm}{\prod_{j=1}^m}

\newcommand\set[1]{\ensuremath{\{#1\}}}

\newcommand\bigpar[1]{\bigl(#1\bigr)}

\newcommand\lrpar[1]{\left(#1\right)}

\newcommand\lrabs[1]{\left|#1\right|}
\def\rompar(#1){\textup(#1\textup)}    

\def\xexp(#1){e^{#1}}

\newcommand\setn{\set{1,\dots,n}}
\newcommand\ntoo{\ensuremath{{n\to\infty}}}

\newcommand\ktoo{\ensuremath{{k\to\infty}}}

\newcommand\punkt[1]{\if.#1\else.\spacefactor1000\fi{#1}}
\newcommand\iid{i.i.d\punkt}
\newcommand\ie{i.e\punkt}
\newcommand\eg{e.g\punkt}

\newcommand{\as}{a.s\punkt}
\newcommand{\aex}{a.e\punkt}

\newcommand{\tend}{\longrightarrow}
\newcommand\dto{\overset{\mathrm{d}}{\tend}}
\newcommand\pto{\overset{\mathrm{p}}{\tend}}
\newcommand\asto{\overset{\mathrm{a.s.}}{\tend}}
\newcommand\eqd{\overset{\mathrm{d}}{=}}

\newcounter{CC}
\newcounter{cc}

\newcommand\E{\operatorname{\mathbb E{}}}
\renewcommand\P{\operatorname{\mathbb P{}}}

\newcommand\Bi{\operatorname{Bi}}

\newcommand\Be{\operatorname{Be}}

\newcommand\fall[2]{\lrpar{#1}_{#2}}

\newcommand\gd{\delta}

\newcommand\gf{\varphi}

\newcommand\gG{\Gamma}

\newcommand\gl{\lambda}

\newcommand\gs{\sigma}

\newcommand\gth{\theta}
\newcommand\eps{\varepsilon}

\renewcommand\phi{\xxx}  

\newcommand\cF{\mathcal F}

\newcommand\cL{{\mathcal L}}

\newcommand\cS{{\mathcal S}}

\newcommand\cU{{\mathcal U}}

\renewcommand{\=}{:=}

\newcommand\intoi{\int_0^1}

\newcommand\oi{[0,1]}

\newcommand\setoi{\set{0,1}}

\newcommand\dd{\,\mathrm{d}}

\newcommand\nioo{_{n=1}^\infty}
\newcommand\nn{[n]}

\newcommand\goo{G_\infty}
\newcommand\restr[1]{|_{#1}}
\newcommand\U{\operatorname{U}}
\newcommand\sss{\cS}
\newcommand\sssmu{(\cS,\mu)}

\newcommand{\ps}{probability space}
\newcommand\gh{graph homomorphism}
\newcommand\nge{n_{<}}
\newcommand\dinj{d^-_j}
\newcommand\oivalued{$0/1$-valued}

\newcommand{\Polya}{P\'olya}

\newcommand\ER{Erd\H os--R\'enyi}
\newcommand{\Lovasz}{Lov\'asz}

\begin{document}
\title[Graph limits of growing sequences of random graphs]
{An example of graph limits of growing sequences of random graphs}

\date{18 June, 2012}

\author{Svante Janson}
\address{Department of Mathematics, Uppsala University, PO Box 480,
SE-751~06 Uppsala, Sweden}
\email{svante.janson@math.uu.se}

\author{Simone Severini}
\address{Department of Computer Science and Department of Physics \& Astronomy,
University College London, Gower Street,
WC1E 6BT London, UK}
\email{s.severini@ucl.ac.k}


\begin{comment}  
05 Combinatorics
05C Graph theory [For applications of graphs, see 68R10, 90C35, 94C15]
05C05 Trees
05C07 Vertex degrees
05C35 Extremal problems [See also 90C35]
05C40 Connectivity
05C65 Hypergraphs
05C80 Random graphs
05C90 Applications
05C99 None of the above, but in this section

60 Probability theory and stochastic processes
60C Combinatorial probability
60C05 Combinatorial probability

60F Limit theorems [See also 28Dxx, 60B12]
60F05 Central limit and other weak theorems
60F17 Functional limit theorems; invariance principles
\end{comment}

\begin{abstract}
We consider a class of growing random graphs obtained by creating vertices sequentially one by one: at each step, we choose uniformly the neighbours of the newly created vertex; its degree is a random variable with a fixed but arbitrary distribution, depending on the number of existing vertices. Examples from this class turn out to be the \ER{} random graph, a natural random threshold graph, \emph{etc}. By working with the notion of graph limits, we define a kernel which, under certain conditions, is the limit of the growing random graph. Moreover, for a subclass of models, the growing graph on any given $n$ vertices has the same distribution as the random graph with $n$ vertices that the kernel defines. The motivation stems from a model of graph growth whose attachment mechanism does not require information about properties of the graph at each iteration.
\end{abstract}

\maketitle

\section{Introduction}\label{S:intro0}

Many models of randomly grown graphs have been studied during the recent
years in the attempt of reproducing characteristic properties of natural and
engineered networks. For example, it is well-known that the
power law (Zipf's law)  on
the degree distribution observed for many real-world networks can occur as a
result of
preferential attachment following some local rule (see, \eg{} Mitzenmacher
\cite{Mitzenmacher} and Durrett \cite{Durrett}).

We may initially distinguish between two types of growth depending on
whether the random steps require or do not require local knowledge of the
graph. Of course, preferential attachment requires local knowledge either
available for free or provided by some dynamics that generates it, for
example, a random walk. Such a distinction is meaningful because it helps
to isolate the type information needed for the construction of specific
network ensembles. Once we have assumed no knowledge, we may further
distinguish between rewiring schemes acting on the whole set of vertices and
mechanisms concerned only with the lastly added vertex. This latter scenario
is considered in the present note.

We grow graphs by attaching vertices one by one: at each step, the
neighbours of the new vertex are chosen uniformly; and the number is a random
variable with a fixed but arbitrary distribution depending on the number of
vertices already present. This mechanism reflects
the idea that the graph is constructed by an agent without any kind of
knowledge of the graph, apart from the labels of the vertices. The role of
the agent is to attach vertices according to the chosen distribution.

We study examples of growing sequences of these random graphs within the
framework of graph limits. (See Lov\'asz \cite{Lov08}; for additional
references and basic definitions see Section \ref{Slim} below.) Every convergent
sequence of growing graphs, where ``convergent" means Cauchy in a specific
metric, has a limit which can be represented in
the form of a symmetric measurable function in two
variables also called a \emph{graphon}. The notion of graph limits has been
central to a general theory of parameter testing as developed by Borgs
\emph{et al.} \cite{BorgsPT}. The wider perspective of graph limits is to
propose an approximation theory of graphs. This would help to study large
graphs/networks by looking at the the proportion of copies of any fixed
graph as a subgraph.

Section \ref{S:intro} defines our construction and lists some of its natural
examples.
Section \ref{Srelated} recasts a special case of the construction in terms
of a certain infinite
random graph.
Section \ref{Slim} gives the necessary definitions concerned with graph
limits and kernels.
Section \ref{Smain} contains the main result.
Section \ref{Sfurther} states
further remarks and formulates several open problems.

\section{Preliminaries}\label{S:intro}

Consider a growing sequence of random graphs $(G_n)\nioo$ defined by the
following Markov process:
\begin{construction}\label{D1}
For each $n\ge1$, let
$\nu_n$ be a given probability distribution
on \set{0,\dots,n-1}.
Construct the random graphs $G_1,G_2,\dots$ as follows.
\begin{romenumerate}
\item
$G_1=K_1$, the graph with a single vertex.
\item
For $n\ge2$,
let $D_{n}$ be a random variable with distribution $\nu_{n}$ and
construct $G_{n}$ by adding a new vertex to $G_{n-1}$ and connecting it to
$D_{n}$ of the previously existing vertices; these vertices are chosen uniformly randomly among all $\binom{n-1}{D_{n}}$ possibilities.
($D_n$ and the choice of vertices are independent of $G_{n-1}$.)
\end{romenumerate}
We may label the vertices $1,2,3,\dots$ in the order they are added, so
$G_n$ has vertex set $\nn\=\setn$.
Since edges are added only incident to the new vertex, and edges are never
removed, we can define the infinite random graph $\goo\=\bigcup\nioo G_n$
with vertex set $[\infty]\=\set{1,2,\dots}$;
then $G_n=G_\infty\restr{\nn}$, the restriction of $\goo$ to the vertex set
$\nn$.
\end{construction}

We may regard $G_n$ as a directed graph by directing each edge towards the
endpoint with largest label. Then $D_k$ is the indegree of vertex $k$ in
$G_n$, for any $n\ge k$. The outdegree of $k$ is 0 in $G_k$, and increases
(weakly) as $n$ grows.

\begin{example}\label{EER}
  Fix $p\in\oi$ and let $\nu_{n}=\Bi(n-1,p)$, $n\ge1$. Then \refCon{D1}
  yields the same result as connecting the new vertex $n$ to each previous
  vertex $i$ with probability $p$, with these events independent for
  $i=1,\dots,n-1$. Hence, $G_n=G(n,p)$, the \ER{} random graph where all edges
  appear independently and with probability $p$ each. This random graph has
  been extensively studied, see \eg{} \cite{Bollobas} and \cite{JLR}.
\end{example}

\begin{example}\label{Ethr}
  Fix $p\in\oi$ and let $\nu_{n}$ be concentrated on
  \set{0,n-1} with
  $\nu_{n}\set{n-1}=\P(D_{n}=n-1)=p$
and
  $\nu_{n}\set0=\P(D_{n}=0)=1-p$.
Thus each new vertex is with probability $p$ joined to all previous
vertices, and with probability $p$ to none. This is an example of a random
threshold graph, see \cite[Section 6.3]{SJ238}, where this $G_n$ is denoted
$T_{n,p}$.

Note that each pair of vertices in $G_n$ is joined by an edge with
probability $p$, just as in \refE{EER}. However, in the present example
these events are not always independent for different pairs.
\end{example}

\begin{example}\label{EU}
Let $\nu_{n}$ be the uniform distribution on $\{0,...,n-1\}$. In this case,
the degree of vertex $n$ in $G_{n}$ is then chosen uniformly
at random among all possibilities.
Thus, if we only consider the number of added edges,
this example uses the \textquotedblleft highest
possible amount of randomness\textquotedblright\ for the construction of the
$n$-th iteration graph,
in the sense that the \emph{entropy} of this number is maximal.
Hence, of all graph
ensembles obtained with Construction \ref{D1}, $G_{n}$ is in some sense the
less predictable one. Note also that the neighbours of $n$ are also chosen at random once
the degree has been determined, again maximising the entropy of this step.
Nevertheless,  as is well-known, the total entropy of the growing random
graph is not maximised by this procedure but by \refE{EER} with $p=1/2$.
\end{example}

The purpose of the present note is to find the limit of the sequence $G_n$
in the sense of {graph limits},
see \refS{Slim}.

All graphs are undirected and finite except when we explicitly say
otherwise.
All unspecified limits are as \ntoo.

\section{A related construction}\label{Srelated}

A class of examples, including the three examples above, can be obtained as
follows.

\begin{construction}\label{D2}
  Let $\nu$ be a given probability measure on $\oi$.
Let $\gth_1,\gth_2,\dots,$ be an \iid{} sequence of random variables with
distribution $\nu$. Then, conditionally given this sequence, let $\goo$ be
the infinite random graph on $[\infty]$ where the edge $\left\{i,j \right\}$ 
appears with probability $\gth_{\max\set{i,j}}$,
and all edges appear independently
(conditionally on $(\gth_j)_{j=1}^\infty$). Further, let
$G_n\=G_\infty\restr{\nn}$.
\end{construction}

If $D_n\=|\set{i<n:in\in E(G_n)}|$, \ie{} the indegree of $n$ if we orient
the edges as above, then $D_n$ conditioned on $(\gth_j)_j$ has the distribution
$\Bi(n-1,\gth_n)$. Hence, the distribution of $D_n$
is a mixture of binomial distributions:
\begin{equation}
  \label{num}
  \begin{split}
\P(D_n=k)
&=\E \Bi(n-1,\gth_n)\set k
=\E\binom{n-1}k\gth_n^k(1-\gth_n)^{n-1-k}
\\&
=\binom{n-1}k\intoi\gth^k(1-\gth)^{n-1-k}\dd\nu(\gth),
\qquad 0\le k\le n-1.	
  \end{split}
\end{equation}
It is obvious that
\refCon{D2} is a special case of \refCon{D1}, with $\nu_n\=\cL(D_n)$ given
by \eqref{num}.

\begin{example}  \label{EER2}
Let $\nu=\gd_p$, a point mass at $p\in\oi$.
Then $\gth_n=p$ and $D_n\sim\Bi(n-1,p)$, in other words, $\nu_n=\Bi(n-1,p)$;
hence \refCon{D2} with this $\nu$ yields the  \ER{} random graph $G(n,p)$ in \refE{EER}.
\end{example}

\begin{example}\label{Ethr2}
  Let $\nu=p\gd_1+(1-p)\gd_0$.
(This is the Bernoulli  distribution $\Be(p)$.)
Then $\gth\in\setoi$, which implies $D_n=(n-1)\gth_n$, and
\begin{align*}
  \P(D_n=n-1)&=\P(\gth_n=1)=p,
\\
 \P(D_n=0)&=\P(\gth_n=0)=1-p.
\end{align*}
Hence, \refCon{D2} yields the random threshold graph in \refE{Ethr}.
\end{example}

\begin{example}\label{EU2}
Let $\nu$ be the uniform distribution on $\oi$; thus $\nu=\gl$, the Lebesgue
measure on $\oi$.
Then each $\gth_n\sim \U(0,1)$ and \eqref{num} yields
by the evaluation of a beta integral,
as is well-known,
\begin{equation}\label{eu2}
  \begin{split}
\nu_n\set k&=\P(D_n=k)
=\binom{n-1}k\intoi\gth^k(1-\gth)^{n-1-k}\dd\gth
\\&
=\binom{n-1}k B(k+1,n-k)=\frac1n,
\qquad 0\le k\le n-1.	
  \end{split}
\end{equation}
Consequently, $\nu_n$ is uniform on \set{0,\dots,n-1}, so \refCon{D2} with $\nu=\gl$
yields the random graphs $G_n$ in \refE{EU}.
\end{example}

\begin{example}
The random graph $G_n$ in Examples \refand{EU}{EU2} can also be constructed
as follows, using some basic results on \Polya--Eggenberger urns.

Recall that a \Polya--Eggenberger urn contains red and black balls; we
  repeatedly draw a ball at random from the urn, and then replace the ball
  together with another ball of the same colour.
If we start the urn with
  one ball of each colour, then the sequence of drawn balls has the same
  distribution as the sequence obtained by first taking a random
  $\gth\sim\U(0,1)$ and then, conditioned on $\gth$, taking a
  sequence of \iid{} balls, each being red with probability $\gth$ and black
  otherwise. This is easily verified by a direct calculation, see
\cite{EggPol1923}, \cite[Theorem 3.1]{Mahmoud} and
\eqref{eu2}. Alternatively, it
is easily seen (again by direct calculation) that the sequence of drawn
balls is exchangeable. By de Finetti's theorem
(see \eg{} \cite[Theorem 1.2]{Mahmoud} or, in a more general version,
\cite[Theorem 11.10]{Kallenberg}),
there exists a random variable $\gth$ with values in
$\oi$ such that conditioned on $\gth$, the sequence of drawn balls is \iid{}
with each ball being red with probability $\gth$.
The law of large numbers
yields $R_n/n\asto\gth$, where $R_n$ is the number of red balls drawn in the
first $n$ draws. To see the representation above, with $\gth\sim\U(0,1)$,
it thus suffices to show that $R_n/n\dto\U(0,1)$, see
\cite{EggPol1923}, \cite[Exercise 3.4]{Mahmoud}.

The sequence of the first $n-1$ drawn balls in this urn thus has the same
distribution as the sequences of indicators of edges $\left\{i,n\right\}$, $i=1,\dots,n-1$
in the random graph $G_n$ in \refE{EU2}. (We translate
$\textit{red}=1$ and $\textit{black}=0$.) The graph $G_n$ is therefore described
by a sequence of (finite) draws from \Polya--Eggenberger urns, independent of each
other.
This can be formulated as the following, rather curious, construction:

Start with vertices \set{-1,0,1,2,3,\dots}. Connect $0$ to all other
vertices (except $-1$), but do not connect $-1$ to any vertex.
For each $k\ge1$, consider $i=1,\dots,k-1$ in order; for each $i<k$ pick a
random $j$ in $-1,0,\dots,i-1$ (uniformly and independent of everything
else), and add an
edge $\left\{i,k\right\}$ if and only if there
already is an edge $\left\{j,k\right\}$. The sequence of edge indicators $\left\{i,k\right\}$,
$i=1,\dots,k-1$, then forms a \Polya--Eggenberger sequence as above, for each
$k$. Consequently, if we discard vertices 0 and $-1$ at the end,
the random graph constructed in this way equals $\goo$ in
\refE{EU2}, and we obtain $G_n$ if we do the same construction for
$k=1,\dots,n$.
\end{example}

We note the following  consequence of the law of large numbers.
\begin{lemma}
  \label{L1}
Let $\nu$ be a probability measure on $\oi$, and let $D_n$ have the mixed
binomial distribution in \eqref{num}. Then $D_n/n\dto \nu$ as \ntoo.
\end{lemma}
\begin{proof}
  Since $D_n$ conditioned on $\gth_n$ has the distribution $\Bi(n-1,\gth_n)$,
we have the law of large numbers
$D_n/(n-1)-\gth_n\pto 0$ as \ntoo. (For example, by computing the
variance.) The result follows since $\gth_n\sim\nu$.
\end{proof}

\section{Graph limits and kernels}\label{Slim}

We assume that the reader is familiar with the theory of graph limits
developed in
\citet{LSz} and
 \citet{BCLSV1,BCLSV2},
see also
\eg{}
\citet{Austin},
\citet{BRmetrics},
\citet{BCL:unique},
Lov\'asz \cite{Lov08},
\citet{SJ209},
\citet{SJ249}.
We recall only a few definition; these will help to fix our notation.

If $F$ and $G$ are finite graphs, let $t(F,G)$ be the probability that a
random mapping $\gf:V(F)\to V(G)$ is a \emph{\gh}, \ie,
satisfies $\gf(i)\sim\gf(j)$ in $G$ whenever
$i\sim j$ in $F$. We say that a sequence $(G_n)$ of graphs with $|G_n|\to\infty$
\emph{converges} if $\lim_\ntoo t(F,G_n)$ exists for every graph $F$.

\subsection*{Graph limits}
The \emph{graph limits} are objects in a suitable space
defined such that each convergent sequence of graphs has a graph limit as
its limit.
If $\gG$ is a graph limit, then $t(F,\gG)$ is defined for every graph $F$,
and a sequence of graphs $G_n$ with $|G_n|\to\infty$ converges to $\gG$ if
and only if $t(F,G_n)\to t(F,\gG)$ for every $F$.
Hence a graph limit $\gG$ is determined by the numbers $t(F,\gG)\in\oi$ for
graphs $F$.
Formally, the graph limits may be defined as equivalence classes of
convergent sequences of graphs, or as suitable families $(t_F)_{F\in\cU}$ of
numbers, where $\cU$ is the set of graphs. The graph limits can be equivalently defined as
classes of kernels, as we do below. This distinction is immaterial. We tacitly refer to unlabelled graphs. 

It is important that the set of all graphs together with all graph limits is a
compact metric space.

\subsection*{Kernels}
Let $(\sss,\cF,\mu)$ be a probability space.
(We usually denote this space simply by $\sss$ or $(\sss,\mu)$,
with $\cF$ and perhaps $\mu$ being clear from the context.)
A \emph{kernel} or \emph{graphon} on $\sssmu$ is
a measurable symmetric function $W:\sss^2\to \oi$. We will consider graphons with codomain $\oi^2$. For this specific setting see \eg{} \citet{BCLSVgr}.

The basic fact is that every kernel $W$ on a probability space $\sssmu$
defines a graph limit $\gG_W$. Conversely, every graph limit equals $\gG_W$ for some kernel $W$. 
We say that \emph{the
graph limit is represented by the kernel $W$}.
Note that $\gG_W$ implicitly depends on $\sss$ and $\mu$ as well as on $W$.
However, such representations of graph limits are not unique.
We say that two kernels $W_1$ and $W_2$, possibly on different probability spaces,
are \emph{equivalent} 
if they represent the same graph limit, \ie, if $\gG_{W_1}=\gG_{W_2}$.
Since every kernel is equivalent to some kernel on $\oi$, every graph limit
may be represented by a kernel $W$ on $\oi$, equipped with Lebesgue measure $\gl$, but even then $W$ is not unique.
Detailed results are in \citet{BCL:unique}, \citet{BRmetrics} and \citet{SJ249}.

If $G_n$ is a sequence of graphs with $G_n\to\gG_W$, for some kernel $W$, we
also write $G_n\to W$.

\subsection*{Random graphs}

Let $W$ be a kernel, defined on a \ps{} $\sssmu$.
We define a random graph $G(n,W)$ with vertex set $[n]$, for $1\le n\le\infty$,
by first taking an \iid{} sequence $\set{X_i}_{i=1}^n$ of random points in
$\sss$ with the
distribution $\mu$, and then, given this sequence, letting $\left\{i,j\right\}$ be an edge
in $G(n,W)$ with probability $W(X_i,X_j)$.
For a given sequence $(X_i)_i$,
this is done independently for
all pairs $(i,j)\in[n]^2$ with $i<j$.
Note that we may construct $G(n,W)$ for all $n$ by first constructing
$G(\infty,W)$ and then taking the subgraph induced by the first $n$ vertices.
A fundamental result is that
for every kernel $W$, $G(n,W)\to W$ a.s.

Furthermore, two kernels $W_1$ and $W_2$ are equivalent,
\ie{}  $\gG_{W_1}=\gG_{W_2}$,
if and only if $G(n,W_1)\eqd G(n,W_2)$ for every finite $n$, and then also for $n=\infty$.

\section{Main results}\label{Smain}

Given a probability measure $\nu$ on $\oi$, let
$\mu=\mu_\nu\=\nu\times\gl$ be a measure on the
product space $\sss\=\oi^2$. Define the kernel $W:\sss^2\to\oi$ by
\begin{equation}\label{w}
  W\bigpar{(s_1,t_1),(s_2,t_2)}
\=
\begin{cases}
  s_2, & \text{if } t_1<t_2;\\
  s_1, & \text{if } t_1>t_2.
\end{cases}
\end{equation}
We may define $  W\bigpar{(s_1,t_1),(s_2,t_2)}\=0$ if $t_1=t_2$; this is
not important since it really is sufficient to have $W$ defined $\mu$-almost
everywhere.

\begin{theorem}\label{TD2}
Let $\nu$ be a probability measure on $\oi$, and let $1\le n<\infty$.
The random graph\/ $G_n$ defined by \refCon{D2}
and
 the random graph $G(n,W)$ defined by the kernel\/ $W$ in \eqref{w} on the
  probability space $(\sss,\mu_\nu)$ are,
regarded as unlabelled graphs, equal in the sense that they have the same
distribution.
\end{theorem}

\begin{remark}
  We have to regard the graphs as unlabelled here, since the vertices in
  $G_n$ are (in general) not equivalent, while they are in $G(n,W)$.
For example, in \refE{Ethr2}, the edges in $G_n$
incident to vertex 1 appear
independently of each other, so the degree of 1 has distribution
$\Bi(n-1,p)$,
while the degree of $n$ is $D_n$, which is 0 or $n-1$.

If we prefer to consider labelled graphs, the correct conclusion is that
$G_n$ with a (uniform) random relabelling of the vertices has the same
distribution as $G(n,W)$, for any finite $n$.
\end{remark}

\begin{remark}
  Similarly, the conclusion of \refT{TD2} fails for $n=\infty$. Consider
  again \refE{Ethr2}. It is easily verified that in $\goo$ there is
  a pair of vertices $i$ and $j$ with the same closed neighbourhoods $\bar
  N(i)$ and $\bar N(j)$ (for example, vertices 1 and 2. In fact, there are
  \as{} infinitely many such pairs), while there is \as{} no such pair in
  $G(\infty,W)$.
\end{remark}

\begin{proof}[Proof of \refT{TD2}]
  Let $X_i=(\xi_i,\eta_i)$, $i=1,2,3,\dots$, be \iid{} points in
  $\sss=\oi^2$ with distribution $\mu_\nu$;
thus each $\xi_i$ has distribution
  $\nu$ and $\eta_i\sim\U(0,1)$, and all $\xi_i,\eta_i$ are independent.

The numbers $\eta_1,\dots,\eta_n$ are \as{} distinct. Order them in
increasing order as $\eta_{i_1}<\eta_{i_2}<\dots<\eta_{i_n}$, and let
$\gth_k\=\xi_{i_k}$. Then $\gth_1,\dots,\gth_n$ are \iid{} with distribution
$\nu$, and $(\gth_i)_{i=1}^n$ is independent of the random permutation
$(i_1,\dots,i_n)$.

Conditioned on $(X_1,\dots,X_n)$,
the edges in $G(n,W)$ appear independently, and the probability of an edge
between $i_j$ and $i_k$, with $j<k$, is
$W(X_{i_j},X_{i_k})=\xi_{i_k}=\gth_k$. Thus, given $(i_1,\dots,i_n)$,
$G(n,W)$ has the same distribution as $G_n$ in \refCon{D2} after the
relabelling $k\mapsto i_k$. Hence, $G(n,W)$ has the same distribution as
$G_n$ with a uniform random relabelling. Consequently, $G(n,W)\eqd G_n$ as
unlabelled graphs.
\end{proof}

\begin{theorem}\label{TD2lim}
  If $G_n$ is defined by \refCon{D2} for some probability measure $\nu$ on
  $\oi$, then $G_n\pto \gG_\nu$ as \ntoo, where
  $\gG_\nu$ is  the graph limit defined by the kernel $W$ in \eqref{w}
  on the \ps{} $(\oi^2,\mu_\nu)$.
\end{theorem}

\begin{proof}
  An immediate consequence of \refT{TD2} and $G(n,W)\pto \gG_W=\gG_\nu$.
\end{proof}

We have a similar result for the more general construction
\refCon{D1}, provided the distributions $\nu_n$ converge to $\nu$ after
rescaling by $n$ (or $n-1$).

\begin{theorem}\label{TD1lim}
  Let $G_n$ be defined by \refCon{D1} for some probability measures $\nu_n$,
  and suppose that $D_n/n\dto\nu$ as \ntoo{} for some probability measure
  $\nu$ on $\oi$, where $D_n\sim \nu_n$.
Then $G_n\pto \gG_\nu$ as \ntoo, where
  $\gG_\nu$ is  the graph limit defined by the kernel $W$ in \eqref{w}
  on the \ps{} $(\oi^2,\mu_\nu)$.
\end{theorem}

\begin{proof}
  If $F$ and $G$ are labelled graphs, let
$n(F,G)$ be the number of \gh{s} $\gf:F\to G$; thus $t(F,G)=n(F,G)/|G|^{|F|}$.
Further, let $\nge(F,G)$ be the number of \gh{s} $\gf:F\to G$ that are
increasing, \ie, $\gf(i)<\gf(j)$ when $i<j$, and let $n_0(F,G)$ be the
number of \gh{s} $F\to G$ that are not injective.

Let $F$ be a fixed graph with vertices labelled $1,\dots,m=|F|$.
If $\gs$ is a permutation of $[m]$, let $F_\gs$ be $F$ relabelled by
$i\mapsto\gs(i)$. For any labelled graph $G$,
\begin{equation}
  \label{nfg}
n(F,G)=
\sum_\gs \nge(F_\gs,G)+n_0(F,G),
\end{equation}
since an injective map $V(F)\to V(G)$ is increasing as a map $F_\gs\to G$
for exactly one permutation $\gs$.

Fix a permutation $\gs$ and consider $\nge(F_\gs,G_n)$, with $G_n$ as in
\refCon{D1}.
We regard $F_\gs$ as a directed graph by directing each edge towards the
endpoint with the largest label.
Let $\dinj\=|\set{i<j:\left\{i,j\right\}\in E(F_\gs)}|$ be the indegree in $F_\gs$ of $j\in[m]$.

Let $\gf:[m]\to\nn$ be an increasing map. Then $\gf$ is a \gh{} $F_\gs\to
G_n$ if and only if, for each $j=1,\dots,m$, $G_n$ contains the $\dinj$
edges $\left\{\gf(i),\gf(j)\right\}$ for $i<j$ with $\left\{i,j \right\}$ $\in E(F_\gs)$.
Conditioned on the indegrees $D_1,\dots,D_n$ in $G_n$, this happens with
probability
\begin{equation}
  \prod_{j=1}^m \frac{\binom{D_{\gf(j)}}{\dinj}}{\binom{\gf(j)-1}{\dinj}}
=
  \prod_{j=1}^m \frac{\fall{D_{\gf(j)}}{\dinj}}{\fall{\gf(j)-1}{\dinj}}.
\end{equation}
Hence, taking the expectation, summing over all $\gf$,
and using the independence of $D_1,\dots,D_n$,
\begin{equation}\label{sofie}
\E \nge(F_\gs,G_n)
=\sum_{1\le\gf(1)<\dots<\gf(m)\le n}
  \prod_{j=1}^m \frac{\E\fall{D_{\gf(j)}}{\dinj}}{\fall{\gf(j)-1}{\dinj}}.
\end{equation}

By assumption, $D_k/k\dto \nu$ as $k\to\infty$. By dominated
convergence, since $0\le D_k/k\le 1$, we have
\begin{equation}
  \frac{\E D_k^d}{k^d}\to M_d\=\intoi x^d\dd\nu(x),
\qquad \ktoo,
\end{equation}
for every $d\ge0$. 
Hence also
\begin{equation}\label{jesper}
  \frac{\E \fall{D_k}d}{\fall{k-1}d}
=  \frac{\E D_k^d+O(k^{d-1})}{k^d+O(k^{d-1})}
\to M_d,
\qquad \ktoo.
\end{equation}
Let $\eps>0$, it follows from \eqref{jesper} that there exists $n_\eps$ such
that if $\gf(1)\ge n_\eps$, then the product in \eqref{sofie} differs by at
most $\eps$ from $\prodjm M_{\dinj}$. For smaller $\gf(1)$ we use the fact
that the product is bounded by 1. The total number of terms in the sum in
\eqref{sofie} is
$\binom nm$, of which $O(n^{m-1})$ have $\gf(1)<n_\eps$, and thus we obtain
\begin{equation}
\lrabs{\E \nge(F_\gs,G_n)-\binom nm \prodjm M_{\dinj}}
\le \eps\binom nm + O(n^{m-1}),
\end{equation}
which implies,
\begin{equation}\label{em}
\E \nge(F_\gs,G_n)
=\binom nm \prodjm M_{\dinj} + o(n^{m})
=\frac{n^m}{m!} \prodjm M_{\dinj} + o(n^{m}),
\end{equation}
since $\eps>0$ is arbitrary.

We have so far considered a fixed $\gs$, but we now sum \eqref{em} over all
$\gs$ and use \eqref{nfg}. Since $n_0(F,G_n)=O(n^{m-1})$,
\begin{equation}\label{magnus}
\E  n(F,G_n)=t_F n^m+o(n^m)
\end{equation}
for some constant $t_F$ depending on $F$ and $\nu$. We have
$m!\,t_F=\sum_\gs \prodjm M_{\dinj}$, where $\dinj$ depends on $F$ and
$\gs$.

Since $t(F,G_n)=n(F,G_n)/n^m$, \eqref{magnus} is the same as
\begin{equation}
 \E t(F,G_n)\to t_F.
\end{equation}
We have proved this for any graph $F$, and it follows by
\cite[Corollary 3.2]{SJ209} that $G_n\pto \gG$ for some graph limit $\gG$.

It remains to identify the limit $\gG$ as $\gG_\nu$. We have proved that
$t_F$ for each graph $F$,
and thus the limit $\gG$, depends on $\nu$ but not otherwise on the
distributions $\nu_n$. For a given distribution $\nu$, we consider
\refCon{D2}, which is a special case of \refCon{D1} with $\nu_n$ the mixture
of binomial distributions given by
\eqref{num}. By \refL{L1}, we have $D_n/n\dto\nu$. We are then in the setting
of the present theorem and the proof above shows $G_n\pto \gG$. On the
other hand, \refT{TD2lim} shows $G_n\pto\gG_\nu$. Hence, $\gG=\gG_\nu$.
\end{proof}

\section{Further comments and open problems}\label{Sfurther}

We have found the limit of the random sequence $G_n$ as a graph limit
defined by a kernel on $(\oi^2,\nu\times\gl)$.
It is easy to find an equivalent kernel on $(\oi^2,\gl\times\gl)$:
Let $\psi:\oi\to\oi$ be the right-continuous inverse of the distribution
function of $\nu$. If $U\sim\U(0,1)$, then $\psi(U)\sim\nu$.
We define $W_\nu$ as the pullback of $W$ via the map
$(s,t)\mapsto(\psi(s),t)$, \ie,
\begin{equation}\label{w2}
  W_\nu\bigpar{(s_1,t_1),(s_2,t_2)}
\=
W\bigpar{(\psi(s_1),t_1),(\psi(s_2),t_2)}
=
\begin{cases}
  \psi(s_2), & \text{if } t_1<t_2;\\
  \psi(s_1), & \text{if } t_1>t_2.
\end{cases}
\end{equation}
Then $W_\nu$ is a kernel on $(\oi^2,\gl^2)$ which is equivalent to $W$ on
$(\oi^2,\nu\times\gl)$; thus we also have $G_n\pto W_\nu$ under the
conditions of Theorem \ref{TD2lim} or \refT{TD2lim}. However, it is at least sometimes possible to find simpler representations.

\begin{example}
  \label{EER3}
In Example \ref{EER} and \ref{EER2}, $\nu=\gd_p$ and $\psi(s)=p$ for all
$s$; thus $W_\nu=p$ is constant. (Similarly, $W=p$ a.e.\ with respect to
$\mu_\nu$.) In fact, as is well known, the graph limit of $G(n,p)$ is
represented by the constant kernel $p$ on any \ps. (Conversely, any
representing kernel equals $p$ \as, see \cite[Corollary 8.12]{SJ249}.)
\end{example}

\begin{example}
  In Example \ref{Ethr} and \ref{Ethr2}, $\nu$ is concentrated on \setoi, so
  $\mu_\nu$ is concentrated on $\setoi\times\oi$.
In particular, the kernel $W$ is \aex{} \oivalued. (This is a general
property of kernels representing limits of threshold graphs; see
\cite{SJ238} and \cite[Section 9]{SJ249}.)

The representation theorem in \cite{SJ238} for general limits of threshold
graphs yields a kernel on \oi. (This kernel is monotone, and this property
makes it uniquely determined  a.e.)
In the present case, the kernel is the indicator function of the
quadrilateral $S_p$ having vertices in $(0,1)$, $(1-p,1-p)$, $(1,0)$ and
$(1,1)$, see \cite[Section 6]{SJ238}. Denote this kernel by $W'$.

It is easy to find a relation between the two representations. Let
$\gf:\oi\to\setoi\times\oi$ be defined by $\gf(x)=(0,1-x/(1-p))$ for
$0\le x\le1-p$ and $\gf(x)=(1,(x-1+p)/p)$ for $1-p<x\le 1$. Then $\gf$ is
measure preserving $(\oi,\gl)\to(\oi^2,\mu_\nu)$ and $W'(x,y)$ is the pullback
$W(\gf(x),\gf(y))$ of $W$.
\end{example}

As said in \refS{Slim}, it is always possible to find an
equivalent kernel on $\oi$. In the two examples above, there are simple and
natural choices of such kernels. However, in Example \refand{EU}{EU2} we do
not know any natural kernel on $\oi$ representing the limit.

\begin{problem}
  Find a natural kernel on $\oi$ representing the limit in \refE{EU}, \ie, a
  natural kernel on $\oi$ that is equivalent to $W$ in \eqref{w} on
  $(\oi^2,\gl^2)$. More generally, find a natural representing kernel on $\oi$ for any $\nu$.
\end{problem}

We close with two different problems inspired by the results above.

\begin{problem}
We have stated Theorems \refand{TD2lim}{TD1lim} with convergence in 
probability. We conjecture that the results are true also almost surely.
\end{problem}

\begin{problem}
  In \refT{TD1lim}, we assume that $D_n/n$ converges in distribution, \ie,
  that the distributions $\nu_n$ converge after rescaling. What happens for
  more general sequences $\nu_n$? Is it possible to characterize the
  sequences $\nu_n$ that give convergence of $G_n$ to some graph limit?
\end{problem}

\begin{ack}
We would like to thank the anonymous referees for their useful comments on the paper. Simone Severini is supported by the Royal Society.
\end{ack}

\newcommand\AAP{\emph{Adv. Appl. Probab.} }
\newcommand\JAP{\emph{J. Appl. Probab.} }
\newcommand\JAMS{\emph{J. \AMS} }
\newcommand\MAMS{\emph{Memoirs \AMS} }
\newcommand\PAMS{\emph{Proc. \AMS} }
\newcommand\TAMS{\emph{Trans. \AMS} }
\newcommand\AnnMS{\emph{Ann. Math. Statist.} }
\newcommand\AnnPr{\emph{Ann. Probab.} }
\newcommand\CPC{\emph{Combin. Probab. Comput.} }
\newcommand\JMAA{\emph{J. Math. Anal. Appl.} }
\newcommand\RSA{\emph{Random Struct. Alg.} }
\newcommand\ZW{\emph{Z. Wahrsch. Verw. Gebiete} }
\newcommand\DMTCS{\jour{Discr. Math. Theor. Comput. Sci.} }

\newcommand\AMS{Amer. Math. Soc.}
\newcommand\Springer{Springer-Verlag}
\newcommand\Wiley{Wiley}

\newcommand\vol{\textbf}
\newcommand\jour{\emph}
\newcommand\book{\emph}
\newcommand\inbook{\emph}
\def\no#1#2,{\unskip#2, no. #1,} 
\newcommand\toappear{\unskip, to appear}

\newcommand\urlsvante{\url{http://www.math.uu.se/~svante/papers/}}
\newcommand\arxiv[1]{\url{arXiv:#1.}}
\newcommand\arXiv{\arxiv}

\def\nobibitem#1\par{}


\begin{thebibliography}{99}

\bibitem[Austin(2008)]{Austin}
T. Austin,
On exchangeable random variables and the statistics of large graphs and
hypergraphs.
\emph{Probability Surveys} \vol5 (2008),  80--145.

\bibitem[Bollob\'as(2001)]{Bollobas}
B. Bollob\'as,
\book{Random Graphs}.
2nd ed., Cambridge Univ. Press,
Cambridge, 2001.

\bibitem[Bollob\'as and Riordan(2009+)]{BRmetrics}
 B.~Bollob\'as, and O.~Riordan,
 Metrics for sparse graphs.
\emph{Surveys in Combinatorics 2009}, LMS Lecture Notes Series 365,
Cambridge Univ. Press, 2009, pp. 211--287.

\bibitem[Borgs, Chayes and Lov\'asz(2010)]{BCL:unique}
C.~Borgs, J. T.~Chayes, and L.~Lov\'asz,
Moments of two-variable functions and the uniqueness of graph limits.
\emph{Geom. Funct. Anal.}  \vol{19} (2010),  no. 6, 1597--1619.

\bibitem[Borgs, Chayes, Lov\'asz, S\'os and Vesztergombi(2008)]{BCLSV1}
  C.~Borgs, J. T.~Chayes, L.~Lov\'asz, V. T.~S\'os, and
K.~Vesztergombi,
Convergent sequences of dense graphs I: Subgraph
frequencies, metric properties and testing,
\emph{Advances in Math.} {\bf 219} (2008), 1801--1851.


\bibitem[Borgs, Chayes, \Lovasz, S\'os and Vesztergombi(2007+)]{BCLSV2}
C. Borgs, J.~T. Chayes, L. \Lovasz, V.~T. S\'os, and K. Vesztergombi,
Convergent sequences of dense graphs II: Multiway cuts and statistical physics.
Preprint, 2007.
\url{http://research.microsoft.com/~borgs/}

\bibitem[Borgs, Chayes, Lov\'asz, S\'os and Vesztergombi(2011)]{BorgsPT}
C.~Borgs, J. T.~Chayes, L.~Lov\'asz, V. T.~S\'os, and
K.~Vesztergombi, Graph limits and parameter testing. In \emph{STOC'06: Proceedings
of the 38th Annual ACM Symposium on Theory of Computing}, 261--270.
ACM, New York, 2006.

\bibitem[Borgs, Chayes, Lov\'asz, S\'os and Vesztergombi(2011)]{BCLSVgr}
  C.~Borgs, J. T.~Chayes, L.~Lov\'asz, V. T.~S\'os, and
K.~Vesztergombi, Limits of randomly grown graph sequences,
\emph{Eur. J. Comb.} {\bf 32} (2011), no. 7, 985--999.

\bibitem[Diaconis, Holmes and Janson(2009)]{SJ238}
P. Diaconis, S. Holmes, and S. Janson,
Threshold graph limits and random threshold graphs.
\emph{Internet Mathematics} \vol5 (2009), no. 3, 267--318.

\bibitem[Diaconis and Janson(2008)]{SJ209}
P. Diaconis and S. Janson,
Graph limits and exchangeable random graphs.
\jour{Rend. Mat. Appl. (VII)}
\vol{28} (2008), 33--61.

\bibitem[Durrett(2007)]{Durrett}
R. Durrett,
\book{Random Graph Dynamics}.
Cambridge University Press, Cambridge, 2007.

\bibitem{EggPol1923}
F. Eggenberger and G. P\'olya,
\"Uber die Statistik verketteter Vorg\"ange.
\emph{Zeitschrift Angew. Math. Mech.} \textbf3 (1923), 279--289.

\bibitem[Janson(2010+)]{SJ249}
S. Janson,
Graphons, cut norm and distance, couplings and rearrangements.
Preprint, 2010.
\arxiv{1009.2376}

\bibitem{JLR}
S.~Janson, T.~{\L}uczak, and A.~Ruci{\' n}ski,
\emph{Random Graphs}.
Wiley, New York, 2000.

\bibitem{Kallenberg}
O. Kallenberg,
\book{Foundations of Modern Probability.}
2nd ed., Springer, New York, 2002.

\bibitem[\Lovasz{} (2008)]{Lov08}
L. \Lovasz{},
Very large graphs.
\emph{Current Developments in Mathematics, 2008},
Int. Press, Somerville, MA, 2009, pp. 67--128.

\bibitem[\Lovasz{} and Szegedy(2006)]{LSz}
L. \Lovasz{} and B. Szegedy,
Limits of dense graph sequences.
\emph{J. Comb. Theory Ser. B} \vol{96} (2006), 933--957.

\bibitem{Mahmoud}
H. M. Mahmoud,
\emph{\Polya\ Urn Models}.
CRC Press, Boca Raton, FL, 2009.

\bibitem{Mitzenmacher}
M. Mitzenmacher,
A brief history of generative models for power law and lognormal
  distributions.
\emph{Internet Mathematics}, \vol{1} (2003), no. 2, 226--251.

\nobibitem{Polya1931}
G. P\'olya,
Sur quelques points de la th\'eorie des probabilit\'es.
\emph{Ann. Inst. Poincar\'e} 1,
117--161, 1931.



\end{thebibliography}
\end{document}